\documentclass[leqno,11pt]{amsart}
\usepackage[utf8]{inputenc}

\usepackage{graphicx,color}
\usepackage{latexsym}
\usepackage{graphicx} 
\usepackage{amsthm,amsmath, amssymb,amsfonts,esint} 
\usepackage{layout,textcomp}
\usepackage{enumerate}

\usepackage[colorlinks=true,urlcolor=blue,
citecolor=red,linkcolor=blue,linktocpage,pdfpagelabels,
bookmarksnumbered,bookmarksopen]{hyperref}

\usepackage[english]{babel}
\theoremstyle{plain}

\newtheorem{theorem}{Theorem}[section]

\newtheorem{lemma}[theorem]{Lemma}

\newtheorem{proposition}[theorem]{Proposition}

\theoremstyle{remark}
\newtheorem{remark}[theorem]{Remark}

\numberwithin{equation}{section}

\title[Existence and obstructions for the curvature]{Existence and obstructions for the curvature  on compact manifolds\\ with boundary}

\author[T. Cruz]{Tiarlos Cruz}
\address[T. Cruz]{ Institute of Mathematics, 
	Federal University of Alagoas
	\newline\indent 
	57072-970, Maceió-AL, Brazil}
\email{\href{mailto: cicero.cruz@im.ufal.br}{cicero.cruz@im.ufal.br}}

\author[A. Silva Santos]{Almir Silva Santos}

\address[A. Silva Santos]{Department of Mathematics, 
	Federal University of Sergipe
	\newline\indent 
	49100-000, Sao Cristov\~ao-SE, Brazil}
\email{\href{mailto: almir@mat.ufs.br}{ almir@mat.ufs.br}}

\author[F. Vit\'orio]{Feliciano Vit\'orio}
\address[F. Vit\'orio]{Department of Mathematics, 
	Federal University of Alagoas
	\newline\indent 
	57072-970, Maceió-AL, Brazil}
\email{\href{mailto: cicero.cruz@im.ufal.br}{feliciano@pos.mat.ufal.br}}

\thanks{ TC was partially supported by CNPq grant number 307419/2022-3 and 405468/2021-0. ASS was partially supported by CNPq grant number 403349/2021-4 and 312027/2023-0. The authors were  supported by CNPq grant 408834/2023-4 and FAPEAL, Brazil (Process E:60030.0000002329/2022)}


\begin{document}
\maketitle
\begin{abstract}
We study the set of curvature functions which a given compact manifold with boundary can possess. First, we prove that the  sign  demanded by the Gauss-Bonnet Theorem is a necessary and sufficient condition  for a given  function to be the  geodesic curvature  or the Gaussian curvature of some  {\it conformally equivalent metric}. Our approach allows us to solve problems that are impossible to solve in the pointwise conformal case. Moreover, we obtain a deep and more delicate information on pointwise conformal deformations. We prove new existence and nonexistence results for metrics with prescribed
curvature in the conformal setting, which depend on the Euler characteristic. \end{abstract}


\section{Introduction}

Motivated  by the Uniformization Theorem, M. S. Berger \cite{Be} studied the problem of finding  a Riemannian metric in a given closed surface $M$ with prescribed Gaussian curvature $f\in C^\infty(M)$. He found a sufficient condition when the Euler characteristic is negative, namely, $f<0$. In a series of papers \cite{K,KW1,KW4,KW2,KW3}, J. Kazdan and F. Warner completely solved this problem and investigated the problem of prescribing the Gaussian curvature (if dim $M=2$) and scalar curvature (if dim $M\geq 3$) in a closed manifold. They also addressed the problem of finding a metric that is either
conformally equivalent or pointwise conformal to some prescribed metric $g$ on $M$, such that it has $f$ as its scalar curvature.

We say that two metrics $g$ and $g_0$ are {\it pointwise conformal} if there exists $u\in C^\infty(M)$ such that $g=e^ug_0$, whereas we say they are {\it conformally equivalent} if there exists a diffeomorphism $\varphi$ of $M$ such that $\varphi^*g$ and $g_0$ are pointwise conformal. Note that the first definition is a particular case of the second one. Besides, any two metrics on a closed surface $M$ are pointwise conformal due to the Uniformization Theorem, which no longer holds for dimensions greater than two. The advantage of the conformally equivalent setting is that the diffeomorphism provides flexibility to solve problems, which is impossible in the pointwise conformal case.

In analogy with what happens in the closed case, J. F. Escobar \cite{E2,E1,E3} raised the problem of finding a metric of constant scalar curvature and constant mean curvature on the boundary in a compact Riemannian manifold $M$ with nonempty boundary and dimension $n\geq 3$. Naturally, this leads to the problem of whether we can realize functions, instead of constants,  to be the curvature of some Riemannian metric. In \cite{EOB}, J. F. Escobar worked on this problem in the case where the scalar curvature vanishes identically. He found the necessary and sufficient conditions for a smooth function to be the mean curvature of a pointwise conformal scalar flat metric.

 Although restricting a priori to conformal deformations may seem restrictive, it is
helpful from a geometric viewpoint. It leads us to solve a deceptively innocent-looking nonlinear elliptic equation, whose existence theory has been given using different techniques, including direct methods of calculus of variations, blow-up analysis, and Liouville theorems, see e.g. \cite{ACA,AYM,MR3473446,CY,MR4460167,CR,DMA,MR4546499, MR4517687,XZ}.

Our main task is to investigate conditions for prescribed Gaussian and geodesic curvatures (for $n=2$) or scalar and mean curvatures (for $n\geq 3$) 
 of a metric  conformally
 equivalent to $g$, or even more, pointwise conformal to $g$. With this goal in mind, we introduce the following definitions:
 \begin{eqnarray*}
      \mbox{CE($g$)}&:=&\{f\in C^\infty(M):R_{\overline g}=f\mbox{ and }H_{\overline g}=0\mbox{ for some metric $\overline g$}\\
      & &\mbox{conformally equivalent to $g$}\}.\\
       \mbox{CE$^0$($g$)}&:=&\{h\in C^\infty(\partial M):R_{\overline g}=0\mbox{ and }H_{\overline g}=h\mbox{ for some metric $\overline g$}\\
      & &\mbox{conformally equivalent to $g$}\}.\\
      \mbox{PC($g$)}&:=&\{f\in C^\infty(M):R_{\overline g}=f \mbox{ and }H_{\overline g}=0  \mbox{ for some  $\overline g\in[g]$}\}.\\
      \mbox{PC$^0$($g$)}&:=&\{h\in C^\infty(\partial M):R_{\overline g}=0 \mbox{ and }H_{\overline g}=h \mbox{ for some  $\overline g\in[g]$}\}.
 \end{eqnarray*}
Here, $R_{\overline g}$ and $H_{\overline g}$ are the scalar and the mean curvatures of $\overline g$, respectively, when $n\geq 3$, and they are replaced by the Gaussian and geodesic curvatures, respectively, when $n=2$. The notation $[g]$ denotes the set of Riemannian metrics on $M$  that are pointwise conformal to $g$.

Let $(M, g)$ be a compact Riemannian surface with boundary. The Gauss-Bonnet Theorem states that 
\begin{equation}\label{GB}
\int_MK_gdv+\int_{\partial M}\kappa_gda=2\pi\chi(M),    
\end{equation}
where $K_g$ and $\kappa_g$ are the Gaussian and geodesic curvatures, and $\chi(M)$ is the Euler characteristic of $M.$ 
In particular, this formula gives a necessary sign condition on the Gaussian curvature of a surface or geodesic curvature on the boundary in terms of its Euler characteristic.  For example, if  $K_g\equiv 0$ (resp. $\kappa_g\equiv 0$), by the Gauss–Bonnet theorem,
\begin{eqnarray}\label{C1}
& & \mbox{If $\mathcal{X}(M)>0$, then $\kappa_g$ (resp. $K_g$) is positive somewhere.} \nonumber\\
& & \mbox{If $\mathcal{X}(M)=0,$ then  $\kappa_g$ (resp.  $K_g$) changes sign unless $\kappa_g\equiv 0$.}\\
& &\mbox{If  $\mathcal{X}(M)<0,$ then $\kappa_g$ (resp. $K_g$)  is negative somewhere.}\nonumber
\end{eqnarray}

Our first result gives the converse to this sign condition as follows:
\begin{theorem}\label{teoa} 
Let $(M,g)$ be a  compact Riemannian surface with boundary.  \begin{itemize}
    \item[(a)] If $K_g\equiv 0$, then  $CE^0(g)$ is precisely the set of functions that satisfies the sign condition (\ref{C1}).
    \item[(b)]  If $\kappa_g\equiv 0$, then $CE(g)$ is precisely the set of functions that satisfies the sign condition (\ref{C1}).
\end{itemize}

\end{theorem}

Theorem \ref{teoa} shows that a question, originally proposed by J. Kazdan and F. Warner in \cite[Question 3, p. 16]{KW1} regarding the solution of the prescribed Gaussian curvature problem for a conformally equivalent metric, can be extended to and holds even for surfaces with  boundary.  Consequently, we present extensions for compact surfaces with boundary of Theorems 6.2 and 11.6 of \cite{KW1}, and Theorem 5.6 of \cite{KW2}.

The new existence results for metrics with prescribed curvatures given by Theorem \ref{teoa} are obtained by applying the methods and results from \cite{CV,KW1, KW2}. The method used to prove the Theorem \ref{teoa} consists in defining a map $F:W^{2,p}(M)\to L^p(M)\oplus W^{\frac{1}{2},p}(\partial M)$, such that if $F(u)=(f,h)$, the Gaussian and geodesic curvature of the metric $e^{2u}g$ is $f$ and $h$, respectively. Given two smooth function $f\in C^\infty(M)$ and $h\in C^\infty(\partial M)$, the Theorem \ref{teoa} is proved if show the existence of some function $u$ with $F(u)=(0,\kappa\circ \varphi)$, for item (a), or $F(u)=(f\circ \varphi,0)$, for item (b), for some diffeomorphism $\varphi$. To solve this problem we use the following ingredients: 
\begin{enumerate}
    \item[1.] Due to the celebrated {\it Uniformization Theorem} for surfaces with boundary \cite{OPS}, we can assume that $g$ has constant Gaussian curvature $K_0$ and constant geodesic curvature $\kappa_0$, with $K_0=0$ in item (a), and $\kappa_0=0$ in item (b).
    \item[2.] The \textit{Inverse Function Theorem:} If  $F(u_0)=(K_0,\kappa_0),$ then we can solve $F(u)=(K,\kappa),$ for $(K,\kappa)$ sufficiently close to $(K_0,\kappa_0)\in L^{p}(M)\oplus W^{\frac{1}{2},p}(\partial M)$, where $p>4$, provided that the linearization of $F$ is invertible at $u_0$;
    \item[3.]  An \textit{Approximation Theorem} which gives necessary and sufficient conditions for $(K,\kappa)$ to be in the $L^{p}(M)\oplus W^{\frac{1}{2},p}(\partial M)$  closure of the orbit of a continuous function $(K,\kappa)$ under the group of diffeomorphisms of $M$;
\item[4.] Since $F(u_0)=(K_0,\kappa_0)$ may not be invertible, we prove a  \textit{Perturbation Theorem} which gives the existence of functions arbitrarily close to $u_0$ which the linearization of $F$ at these functions is invertible.
\end{enumerate}

We also obtain a more subtle piece of information on the pointwise conformal deformation. Under the hypothesis of Theorem \ref{teoa}, we show  
  \begin{itemize}
      \item[(1)] $\chi(M)<0$: In contrast with Theorem \ref{teoa}, there exists $\kappa\in C^\infty(\partial M)$ (resp. $K\in C^\infty(M)$) which is negative somewhere and $\kappa\not\in PC^0(g)$ (resp. $K\not\in PC(g)$). See Theorem \ref{teo001}, where we have established necessary and sufficient conditions, in addition to \eqref{C1}, for the solvability of  pointwise conformal problem. This shows that Theorem \ref{teoa} does not holds in the pointwise conformal case and the conformally equivallet approach is more flexible.
      \item[(2)] $\chi(M)=0$: There are necessary and sufficient conditions to a function $\kappa\in C^\infty(\partial M)$ (resp. $K\in C^\infty(M)$) to belongs to $PC^0(g)$ (resp. $PC(g)$), respectively. See Theorem \ref{PCzerocarac}.
      \item[(3)] $\chi(M) > 0$: We introduce new conditions to ensure that functions belong to PC($g$) and PC$^0$($g$). 
      See Section \ref{sec003}.
  \end{itemize}

In the literature, a comprehensive study is presented regarding the sets PC$^0(g)$ and PC$(g)$. For instance, K. C. Chang and J. Q. Liu \cite{MR1417436} established conditions for a positive function $\kappa$, defined on the boundary of a disk, belonging to PC$^0(g)$. Subsequently, P. Liu and W. Huang \cite{MR2103946} investigated this problem when $\kappa$ possesses certain symmetries. This  case was also addressed in \cite{MR2178789, phdthesis-liu}, where a Moser-Trudinger inequality on the boundary of a compact surface was proved. Refer to \cite{MR2274940} as well, where Y. X. Guo and J. Q. Liu performed a blow-up analysis in order to study the function space $\mathrm{PC}^0(g)$. On the other hand, the case PC$(g)$ was treated by S. Y. A. Chang and P. C. Yang \cite{MR925123}. They obtained sufficient conditions on $K$ defined in a domain $D$ on the round sphere, with $\operatorname{Area}(D) < 2 \pi$ or $D$ being a hemisphere. The case $2\pi < \operatorname{Area}(D) < 4\pi$ was addressed by K. Guo and S. Hu in \cite{MR1360827}.

There are few results regarding the problem of prescribing Gaussian and geodesic curvature simultaneously. This problem was studied in \cite{MR4517687} using a variational approach and blow-up analysis, with a dependence on the Euler characteristic. Further insights can be found in \cite{phdthesis, MR4460167, CR, MR4546499}.

Given a Riemannian manifold $(M,g)$ with boundary and dimension $n\geq 3$, we denote by $\mathcal L_g$ the conformal Laplacian $-\frac{4(n-1)}{n-2}\Delta_g +R_{g} $, and by $\mathcal B_g$  the boundary conformal operator $\frac{2}{n-2} \frac{\partial }{\partial \nu_g}+H_{g} $, where $\nu_g$ is  the outward unit normal on $\partial M$. Consider the following two lowest eigenvalues of the boundary linear problem $(\mathcal L_g,\mathcal B_g)$, see Section \ref{sec001},
\begin{equation}\label{eing1}
\mathcal L_g\varphi=\lambda_1(\mathcal L_g)\varphi  \mbox{ in }M\quad\mbox{ and }\quad 
\mathcal B_g\varphi=0 \mbox{ on }\partial M
\end{equation}
and 
\begin{equation}\label{eing2}
\mathcal L_g\varphi=0  \mbox{ in }M\quad \mbox{ and }\quad
\mathcal B_g\varphi=\sigma_1(\mathcal B_g)\varphi \mbox{ on $\partial M$}.
\end{equation}
The signs of $\lambda_1(\mathcal L_g)$ and $\sigma_1(\mathcal B_g)$ are crucial in our study.
 Not only because they are conformal invariant (as discussed in \cite{E1,E3}), but also because we will see that they play a similar role to the sign of the Euler characteristic in our investigation of geodesic and Gaussian curvatures on compact surfaces with boundary.
 
For higher dimensions, we present the following extension to compact manifolds with boundary, building upon the results established by J. Kazdan and F. Warner in \cite[Theorem 3.3]{KW3} and \cite[Theorem 6.2]{KW2}.
\begin{theorem}\label{teob}
Let $(M^n,g)$ be a compact Riemannian manifold with nonempty boundary and dimension  $n\geq 3$.
\begin{itemize}
    \item[(a)] If $\sigma_1(\mathcal B_g)<0$, then CE$^0(g)$ is the set of smooth functions on $\partial M$ that are negative somewhere.
    \item[(b)] If $\sigma_1(\mathcal B_g)=0$ , then CE$^0(g)$  is the set of smooth functions on $\partial M$ that either change sign or are identically zero on $\partial M$.
    \item[(c)] If $\sigma_1(\mathcal B_g)>0$, then CE$^0(g)$ is the set of smooth functions on $\partial M$ that are positive somewhere.
\end{itemize}
\end{theorem}

We observe that similar results can be achieved for $\lambda_1(\mathcal L_g)$.
\begin{theorem}\label{teoc}
Let $(M^n,g)$ be a compact Riemannian manifold with nonempty boundary and dimension  $n\geq 3$.
\begin{itemize}
    \item[(a)] If $\lambda_1(\mathcal L_g)<0$, then CE$(g)$ is the set of smooth functions on $M$ that are negative somewhere.
   \item[(b)] If $\lambda_1(\mathcal L_g)=0$, then CE$(g)$ is the set of smooth functions on $M$ that either change sign or are identically zero on $\partial M$.
    \item[(c)] If $\lambda_1(\mathcal L_g)>0$, then CE$(g)$ is the set of smooth functions on $M$ that are positive somewhere.
\end{itemize}
\end{theorem}

We highlight that the strategy for Theorems \ref{teob} and \ref{teoc} is completely analogous to the case of dimension $n=2.$ However, instead of using the Uniformization Theorem, we utilize the properties of the eigenvalues of $(\mathcal L_g,\mathcal B_g).$

We also present results related to pointwise conformal change. More precisely, given $H \in C^\infty(\partial M)$ with $\int_{\partial M} H da < 0$ (resp. $R \in C^\infty(M)$ with $\int_M R dv < 0$), if $\lambda_1(\mathcal{L}_g) < 0$, we establish a sufficient condition to ensure that $H \in PC^0(g)$ (resp. $R \in PC(g)$). Refer to Theorem \ref{prop002} for details. This pointwise conformal setting has been extensively studied in the last few decades. For a comprehensive overview of the problem, consult \cite[Section 1.3]{phdthesis} and the references therein.

Let us conclude this introduction with a brief description of the structure of this work.  In Section \ref{perturbation}, we present a technical perturbation result that is useful to guarantee the invertibility of a certain operator related to the curvature. In Section  \ref{sec3} we prove  Theorem \ref{teoa} and  address the problem of prescribing the curvatures of surfaces with boundary in a conformal class,  depending on the sign of the Euler characteristic. We also study a higher-order analogue in Section \ref{sec002}, proving Theorems \ref{teob} and \ref{teoc}, and establishing some existence and nonexistence of conformal metrics depending on the sign of the lowest eigenvalues of equations \eqref{eing1} and \eqref{eing2}.

\noindent
{\bf Acknowledgements}.
The authors would like to thank R. L\'opez-Soriano  for theinterest in this work and to help them to improve an earlier version of this paper, as well as to L. Mari for indicating some references and related discussions. Finally, the authors thank the anonymous referee for the valuable suggestions and comments.

\section{A perturbation result}\label{perturbation}

Throughout this paper, we let $(M^n,g)$ be an $n$-dimensional Riemannian   manifold with nonempty boundary $\partial M$. With respect to this metric, the
volume  and area element will be denoted  by $dv$ and $da,$ respectively. If $n=2$ the Gaussian and the geodesic curvature will be denoted by $K_g$ and $\kappa_g$, respectively. While, if $n\geq 3$, the scalar and the mean curvature will be denoted by $R_g$ and $H_g$, respectively.   Given a function $f\in C^{\infty}(M),$  the Laplacian with respect to $g$ will be denoted by $\Delta_g f=\mbox{div}\nabla_g f$.

For $n=2,$ if $\overline g$ is a pointwise conformal metric to $g$  with Gaussian and geodesic curvatures equal to $K$ and $\kappa$, respectively, then we can write $\overline g=e^{2u}g$, and the relation between the Gaussian and geodesic curvature with respect to $g$ and $\overline{g}$ is   given by
\begin{equation}\label{confchange_n=2}
\begin{cases}\mathcal L_gu\equiv-\Delta_g u+K_{g} =Ke^{2u} & \text { in } M, \\ \mathcal B_gu\equiv\dfrac{\partial u}{\partial \nu_{g}}+\kappa_{g} =\kappa e^{u} & \text { on } \partial M.\end{cases}
\end{equation}
Define
\begin{equation}\label{TQ1}
T_1(u):=e^{-2u}\mathcal{L}_gu\quad\mbox{ and }\quad Q_1(u):=e^{-u}\mathcal{B}_gu.
\end{equation}
For $n\geq 3$, if $\overline g$ is a pointwise conformal metric to $g$  with  scalar and mean curvatures equal to $R$ and $H$, respectively, then there exists a smooth function $u>0$ such that  $\overline g=u^{\frac{4}{n-2}}g$. The relation between the scalar and mean curvature with respect to $g$ and $\overline{g}$ is  given by  
\begin{equation}\label{confchange}
\begin{cases}\mathcal L_gu\equiv-\frac{4(n-1)}{n-2}\Delta_g u+R_{g} u=Ru^{\frac{n+2}{n-2}} & \text { in } M, \\ 
\mathcal B_gu\equiv\frac{2}{n-2} \dfrac{\partial u}{\partial \nu_g}+H_{g} u=Hu^{\frac{n}{n-2}} & \text { on } \partial M.\end{cases}
\end{equation}
We also define
\begin{equation}\label{TQ2}
T_2(u):=u^{-a}\mathcal{L}_gu  \quad\mbox{ and }\quad  
Q_2(u):=u^{-b}\mathcal{B}_gu,
\end{equation}
where $a=(n+2)/(n-2)$ and $b=n/(n-2)$.
 Finally,  set
\begin{equation}\label{eq019}
    F_i(u)=(T_i(u),Q_i(u)),\quad i=1,2.
\end{equation}

In order to solve the equation $F_i(u)=(f,h)$ for functions $f$ and $h$, we will consider the linearization of $F_i$ with respect to $u$, which is given by
 $
 F_i'(u) v=(T_i'(u)v,Q_i'(u)v),
 $
where
\begin{equation*}
\quad \begin{cases} T'_1(u)v=-e^{-2 u}[\Delta_g v-2(\Delta_g u-K_g) v]=-e^{-2 u} A_{1}(u) v, \\ 
Q'_1(u)v=-e^{-u}\left(-\dfrac{\partial v}{\partial \nu_g}+\left(\dfrac{\partial u}{\partial \nu_g}+\kappa_g\right)v\right)=- e^{-u}B_1(u)v,\end{cases}
\end{equation*}
for $n=2$, and for $n\geq 3$ and $u>0$ we have
\begin{equation*}
\begin{cases} T_{2}^{\prime}(u) v=-\alpha u^{-a}\left(\Delta_g v-\left(a \dfrac{\Delta_g u}{u}+(a-1) \dfrac{R_g}{\alpha}\right) v\right)=-\alpha u^{-a} A_{2}(u) v, \\ 
Q'_2(u)v=-\beta u^{-b}\left(-\dfrac{\partial v}{\partial \nu_g}+\left(\dfrac{b}{u}\dfrac{\partial u}{\partial \nu_g}+\dfrac{b-1}{\beta}H_g\right)v\right)= -\beta u^{-c}B_2(u)v,\end{cases}
\end{equation*}
where $\beta= 2/(n-2)$ and $\alpha=2(n-1)\beta$.
 It is standard to check that each pair $(A_i(u),B_i(u))$, $i=1,2$, is a  self-adjoint linear elliptic  operator satisfying  $\mbox{ker}\:(A_i(u),B_i(u))=\mbox{ker } F'_i(u).$
Before we state our next lemma we observe that if $v\in \mbox{ker}\:(A_i(u),B_i(u))$ is a nontrivial function, then $v$ is a solution of the problem
$$\left\{\begin{array}{rcll}
A_i(u)v & = & 0\cdot v &\mbox{ in }M \\
    B_i(u)v & = & 0 & \mbox{ on }\partial M,
\end{array}\right.$$
 which is the well-known eigenvalue problem under Robin boundary condition. The first eigenvalue has unit multiplicity and the associated eigenfunction can be taken to be everywhere positive on $M$ (See \cite[Appendix A]{CL}). This implies that $dim\mbox{ ker}\:(A_i(u),B_i(u))=1$.
Now we state the following Perturbation lemma.

\begin{lemma}[Perturbation Lemma]\label{pert}
Let $(M^n,g)$ be a compact Riemannian manifold with nonempty boundary and dimension  $n\geq 2$. Suppose that 
  $$F_i'(u):W^{2,p}(M)\to L^{p}(M)\oplus W^{\frac{1}{2},p}(\partial M)$$  is not invertible. Then there exists a smooth function $z$ such that 
$$
dim\mbox{ ker } F_i'(u+tz)=0,
$$
for $|t|>0$ sufficiently small.
\end{lemma}

\begin{proof}
Suppose that $dim\mbox{ ker }T_i'(u)=1$. Given a smooth function $z$, for each $i=1$ and $2$, there exist normalized functions $v_i(t)$ and numbers $\lambda_i(t)$, such that $v_i(0)$ is a nonzero function on $\mbox{ker }T_i'(u)$, $\lambda_i(0)=0$ and
\begin{equation}\label{eq013}
    \left\{\begin{array}{rcll}
     A_i(u+tz)v_i(t) & = & \lambda_i(t) v_i(t) & \mbox{ in } M \\
     B_i(u+tz)v_i(t) & = & 0 & \mbox{ on } \partial M.
\end{array}\right.
\end{equation}

By \cite[Lemma A.1]{MR3406373} (see also \cite[Lemma A.2]{CL}) the maps $t\mapsto \lambda_i(t)$ and $t\mapsto v_i(t)$ are smooth. We will use the following notation 
$$A_i'(\varphi)=\left.\frac{d}{dt}\right|_{t=0}A_i(u+tz)\varphi\;\;\;\mbox{ and }\;\;\;B_i'(\varphi)=\left.\frac{d^2}{dt^2}\right|_{t=0}B_i(u+tz)\varphi.$$
where $\varphi$ is any smooth function. 
Since $\lambda_i(0)=0$, taking the derivatives of \eqref{eq013} at $t=0$, we get
\begin{equation}\label{eq014}
    A_i'(v_i)+A_i(v_i')  =\lambda_i'v_i \quad\mbox{ and }\quad B_i'(v_i)+B_i(v_i')  =0
\end{equation}

Since $\|v_i\|_{L^2(M)}=1$, then $\langle v_i,v_i' \rangle=0$. At $t=0$ we have $A_i(v_i)=0$ and $B_i(v_i)=0$. Since the pair $(A_i(u),B_i(u))$ is self-adjoint we get 
$\langle v_i,A_i(v_i')\rangle+\langle v_i,B_i(v_i')\rangle=0.$
Therefore, by \eqref{eq014} we obtain
\begin{equation}\label{eq018}
    \lambda_i'  =  \langle A_i'(v_i),v_i \rangle+\langle A_i(v_i'),v_i\rangle=\langle A_i'(v_i),v_i \rangle+\langle B_i'(v_i),v_i \rangle.
\end{equation}

Now we have two cases. 


\noindent{\bf Case 1:} $i=1$. In this case, for any smooth function $\varphi$ we have $A_1'(u)\varphi=-2\varphi\Delta_g z$ and $B_1'(u)\varphi=\frac{\partial z}{\partial\nu_g}\varphi$.
Thus, by \eqref{eq018} at $t=0$  we get
\begin{equation}\label{eq010}
    \lambda_1'=-2\langle \Delta_g z,v_1^2\rangle+\left\langle \frac{\partial z}{\partial\nu_g},v_1^2 \right\rangle.
\end{equation}

Thus, if $v_1(0)$ is not constant we can choose $z=\varphi\Delta_g v_1(0)^2$, where $\varphi$ is a nonnegative smooth function which vanishes in a neighborhood of $\partial M$. This implies that $\lambda_1'=-2\langle \varphi,(\Delta_g v_1(0)^2)^2 \rangle<0$.

Now, suppose $v_1(0)$ is constant. Using integration by parts \eqref{eq010} we get 
$$\lambda_1'=-v_1(0)^2\int_{\partial M}\frac{\partial z}{\partial\nu_g}.$$
It is not difficult to prove the existence of a function $z$ with $\partial z/\partial\nu_g>0$.

\medskip

\noindent{\bf Case 2:} $i=2$. In this case $u>0$ and for any smooth function $\varphi$ we have $A_2'(u)\varphi=-\frac{a\varphi}{u}W(z)$ and $B_2'(u)\varphi=-\frac{b\varphi}{u}V(z),$ 
where $W(z)=\Delta_g z-\frac{\Delta_g u}{u}z$ and $V(z)=\frac{z}{u}\frac{\partial u}{\partial\nu}-\frac{\partial z}{\partial\nu}$.
By \eqref{eq018} at $t=0$ we have
\begin{equation}\label{eq012}
    \lambda_2'=-a\left\langle \frac{v_2^2}{u},W(z) \right\rangle-b\left\langle \frac{v_2^2}{u},V(z) \right\rangle.
\end{equation}

Similar to the case 1, if $W(v_2^2/u)$ is not identically zero, we consider $z=\varphi W(v_2^2/u)$, where $\varphi$ is a nonnegative smooth function which vanishes in a neighborhood of $\partial M$. Using integration by parts we get $\lambda_2'=-a\langle \varphi,W(v_2^2/u)^2\rangle<0$.
If $W(v_2^2/u)\equiv 0$, then by \eqref{eq012} we obtain
$$\lambda'_2=-\int_{\partial M}\left((a-b)\frac{v_2^2}{u}\frac{\partial z}{\partial\nu_g}+z\left(b\frac{v_2^2}{u^2}\frac{\partial u}{\partial\nu_g}-a\frac{\partial}{\partial\nu_g}\left(\frac{v_2^2}{u}\right)\right)\right),$$
with $a-b\not=0$. It is not difficult to find a smooth function $z$ such that $z=0$ on $\partial M$ and $\frac{\partial z}{\partial\nu}<0$ (for instance, a smooth function that coincides in a neighborhood of the boundary with the distance function of a point to the boundary $\partial M$). 
\end{proof}

\begin{theorem}[Perturbation Theorem]\label{perttheo} 
Let $(M^n,g)$ be a Riemannian manifold with nonempty boundary of dimension $n\geq 2$. Then, the operator $$F'_i(u):W^{2,p}(M)\to L^{p}(M)\oplus W^{\frac{1}{2},p}(\partial M)$$  is invertible on an open dense set of functions $u \in C^{2}(M)$, where we assume that $u$ is positive  if $n\geq 3$.
\end{theorem}
\begin{proof}
Since $F_i'(u)$ depends continuously on $u$ (See \cite[Appendix A]{CL}),  the openness assertion follows. By Theorem 2.20 in \cite{Gazzola_Grunau_Sweers_2010}, the operator $F'(u)$ is invertible if and only if $\mbox{ker }F'_i(u)=\mbox{ker }(A_i(u),B_i(u))=0$. The density condition follows by Lemma \ref{pert}.
\end{proof}

Consider a surface with boundary endowed with a flat metric and geodesic boundary. Notice  that $F'_1(0)\cdot v=(-\Delta_g v,\partial v/\partial \nu_g)$ is an example of non-invertible  operator, since its kernel is composed by constants. However,  Perturbation Theorem \ref{perttheo} is useful to guarantees that for any $\varepsilon>0$ there is a smooth function $u_0$ so close to zero that $F'_1(u_0)$ is invertible and $\|F(u_0)\|_{\infty}\leq \varepsilon.$ In that regard, we highlight that  in some situations this Perturbation Theorem is not needed, as for example,  for a surface $M$ with negative Gaussian curvature and geodesic boundary,  since $F'_1(u_0)$ is already invertible.

\section{Curvature functions for surfaces with boundary}\label{sec3}

The main goal of this section is to study the   prescribed  Gaussian and geodesic curvature problem in surfaces with boundary from  a conformal  viewpoint. The focus of our results concerns the case in which at least one of the curvatures is zero. First we will prove  Theorem \ref{teoa}, and then we will address the problem of prescribing the curvatures in a conformal class. 

First we  describe the $L^p$ and the $W^{1/2,p}$ closure of the orbit of functions under the group of diffeomorphisms of $M$.

\begin{lemma}[Approximation Lemma \cite{CV}]\label{appr} Let $(M^n,g)$ be a Riemannian \linebreak manifold with nonempty boundary of dimension $n\geq 2$. 
\begin{itemize}
\item[(a)] Let $f\in C^{\infty}(\partial M)\cap W^{\frac{1}{2},p}(\partial M).$ Given $h\in W^{\frac{1}{2},p}(\partial M),$ if  $\min f\leq h(x)\leq\max f $ on $\partial M$, then for all $\varepsilon>0$ there exists a diffeomorphism  $\varphi$ of $M$ such that, for $p>2n,$ we have  that
$$\|f\circ \varphi-h\|_ {W^{\frac{1}{2},p}(\partial M)}<\varepsilon.$$
\item[(b)] Let $f\in C^{\infty}( M)\cap L^p(M).$ Given $h\in L^p(M),$ if $\min f\leq h(x)\leq\max f $ on $ M$, then for all $\varepsilon>0$ there exists a diffeomorphism  $\varphi$ of $M$ such that, for $p>n,$ we have  that 
$$\|f\circ \varphi-h\|_ {L^{p}( M)}<\varepsilon.$$
\end{itemize} 
\end{lemma}

Another  fundamental tool is the Inverse Function Theorem for Banach Spaces (see, for instance, \cite{MR1864986}). Since the operator $$F_i(u):W^{2,p}(M)\to L^{p}(M)\oplus W^{\frac{1}{2},p}(\partial M),\quad i=1,2,$$
is a $C^1$ map, for $p>2n$, if 
$F'_{i}(u)=(T_i'(u),Q_i'(u))$ is invertible, then there is $\delta>0$ such that given $(f,h)\in C^{\infty}(M)\times C^{\infty}(\partial M)$ with
$$
\|(f,g)-F_i(u))\|_{L^p\oplus W^{\frac{1}{2},p}}<\delta,
$$
there exists $v\in C^{\infty}(M)$ such that $F(v)=(f,h).$

\subsection{Proof of Theorem \ref{teoa}}\label{proof-theorem-1-2}

\begin{proof}[Proof of Theorem \ref{teoa}]
We establish the sufficiency condition for item (a) as the necessity of \eqref{C1} is evident. Let us assume that $\kappa$ satisfies \eqref{C1}. Due to the celebrated Osgood, Phillips, and Sarnak uniformization theorem for surfaces with boundaries \cite{OPS} (also referenced in Brendle \cite{B1,B2}), we can assume that $g$ is a Gauss flat metric with a constant geodesic curvature equal to $c$. In this scenario, we have $F_1(0)=(T_1(0),Q_1(0))=(0,c)$, where $F_1$ is defined in \eqref{eq019}. This trivially implies the result if $\kappa$ is constant.

Now, let us assume that $\kappa$ is not constant. First, suppose that
\begin{equation}\label{ineqn_2}
 \min \kappa <c<\max \kappa.
 \end{equation}

The Perturbation Theorem \ref{perttheo} shows that given any $\varepsilon>0$, there exists a smooth function $u_1$ sufficiently close to $u_0=0$ such that
\begin{equation*}\label{varepsilon}
    \|F_1(0)-F_1(u_1)\|_{\infty}=\|(0,c)-F_1(u_1)\|_{\infty}<\varepsilon,
\end{equation*} 
and $F_{1}'\left(u_{1}\right)$ is invertible. By \eqref{ineqn_2}, for $\varepsilon>0$ small enough we have
$$
\min \kappa<Q_{1}\left(u_{1}\right)<\max \kappa.
$$
If $p>4$, the Inverse Function Theorem \cite{MR1864986} implies that there exists $\delta>0$ such that if $(f,h)\in C^\infty(M)\times C^\infty(\partial M)$ and $\|(f,h)-F_1(u_1)\|_{L^p\oplus W^{\frac{1}{2},p}}<\delta$ then there is $u\in C^\infty(M)$ such that $F_1(u)=(f,h)$. We can take $u$ smooth because the operator is elliptic.

By Lemma \ref{appr} there is a diffeomorphism $\varphi$ of $M$ such that 
$$\left\|\kappa \circ \varphi-Q_{1}\left(u_{1}\right)\right\|_{W^{1/2,p}}<\delta/2.$$ 
Thus, for $\varepsilon<\delta/2$ we have $\|(0,\kappa\circ\varphi)-F_1(u_1)\|_{L^p\oplus W^{\frac{1}{2},p}}<\delta$.
Therefore, we can find a smooth function $u$ such that $F_{1}(u)=(0,\kappa \circ \varphi)$.

Finally, suppose that \eqref{ineqn_2} does not hold. In this case we use \eqref{C1} to conclude that $\kappa$ and $c$ have the same sign at some point of $M$. Thus, there is a constant $m>0$ such that the function $m\kappa$ satisfies \eqref{ineqn_2}. By the previous case, there exists a smooth function $u$ and a diffeomorphism $\varphi$ such that $F_1(u)=(0,m\kappa\circ\varphi)$. Consequently, if we define $v = u + \log m$, then $F_1(v) = (0, \kappa \circ \varphi)$.

Item (b) is similar since  Osgood, Phillips, and Sarnak uniformization theorem also provides a conformal metric of constant Gaussian curvature and geodesic boundary.
\end{proof}

\subsection{Pointwise conformal metrics  in surfaces with boundary}\label{pointn=2}

Contrary to the case of conformally equivalent surfaces, we will divide the pointwise conformal analysis for surfaces with boundary into three parts based on the sign of the Euler characteristic: $\chi(M)=0$, $\chi(M)<0$, and $\chi(M)>0$. In this context, we are extending the results for closed manifolds from \cite{KW1} to manifolds with boundaries.

\subsubsection{{\bf Case 1:} $\chi(M)=0$}

In this case, we will prove a sufficient and necessary condition for functions belonging to PC($g$) or PC$^0$($g$), see definitions (c) and (d) in the introduction. We start the analysis by providing the following result.

\begin{proposition}\label{propn=2}
Let $(M,g)$ be a compact Riemannian surface with nonempty boundary.
\begin{itemize}
    \item[(a)] Given $\kappa\in C^\infty(\partial M)$, with $\kappa\not\equiv 0$, there exists a  solution $u\in C^\infty(M)$ to the problem $\Delta_g u=0$ in $M$ with $ \partial u/\partial \nu_{g}=\kappa e^{u}$  on $\partial M$ if and only if, $\kappa$ changes sign and $\int_{\partial M}\kappa da<0$.
    \item[(b)] Given $K\in C^\infty(M)$, with $K\not\equiv 0$, there exists a  solution $u\in C^\infty(M)$ to the problem $-\Delta_g u=Ke^{2u}$ in $M$ with $ \partial u/\partial \nu_{g}=0$  on $\partial M$ if and only if, $K$ changes sign and $\int_{M}K dv<0$.
\end{itemize}
\end{proposition}
\begin{proof}
Let us prove item (a). Suppose $u\in C^\infty(M)$ is such that $\Delta_gu=0$ in $M$ and $\partial u/\partial\nu_g=\kappa e^u$ on $\partial M$. Using integration by parts we get
$$0=\int_M\Delta_gudv=\int_{\partial M}\frac{\partial u}{\partial\nu_g}da=\int_{\partial M}\kappa e^uda,$$
which implies that $\kappa$ changes sign. Since $u\not\equiv 0$, then
\begin{equation}\label{eq026}
    \int_{\partial M}\kappa da=\int_{\partial M}e^{-u}\frac{\partial u}{\partial\nu_g}da=\int_Me^{-u}\Delta_gudv-\int_Me^{-u}|\nabla_g u|^2dv<0.
\end{equation}

Now, suppose that $\kappa$ changes sign and $\int_{\partial M}\kappa da<0$. Define 
$$\mathcal D:=\left\{w\in W^{1,2}(M):\int_{\partial M}\kappa e^wda=0\quad\mbox{ and }\quad\int_Mwdv=0\right\}.$$
Using the hypothesis, it is not difficulty to show that $\mathcal D\not=\emptyset$ and $w\not\equiv 0$ on $\partial M$ for all $w\in\mathcal D$. Now define the functional $F:\mathcal D\rightarrow\mathbb R$ given by
$$F(w):=\int_M|\nabla_g w|^2dv\geq 0.$$
Let $a:=\displaystyle\inf_{w\in\mathcal D}F(w)$ and let $\{w_k\}$ be a minimizing sequence in $\mathcal D$. Note that
$$\|w_k\|^2_{W_{1,2}}=F(w_k)+\int_Mw_k^2dv.$$
By Poincaré-Sobolev inequality, there exists a constant $c>0$ such that for any $w\in W^{1,2}(M)$ with $\int_Mwdv=0$, it holds 
\begin{equation}\label{eq021}
    \int_Mw^2dv\leq c\int_M|\nabla_g w|^2dv,
\end{equation}
Thus
$$\|w_k\|^2_{W_{1,2}}\leq cF(w_k),$$
which implies that $\{w_k\}$ is bounded in $W^{1,2}(M)$. Since a ball in any Hilbert space is weakly compact, there exists a function $w\in W^{1,2}(M)$ and a subsequence of $\{w_k\}$, still denoted by $\{w_k\}$, converging weakly to $w$. Thus $\int_Mwdv=0$.

By \cite[Proposition 3.16]{phdthesis} we have that $e^{w_k}\rightarrow e^w$ in $L^2(\partial M)$. This implies that $\int_{\partial M}\kappa e^w=0$. Therefore $w\in \mathcal D$. Since $w_k\rightharpoonup w$ in $W^{1,2}(M)$, it is well-known that $\|w\|_{W^{1,2}}\leq\liminf\|w_k\|_{W^{1,2}}$. Since $W^{1,2}(M)$ is compactly embedding in $L^2(M)$, then $w_k\rightarrow w$ in $L^2(M)$. This implies that $F(w)\leq\liminf F(w_k)$. Therefore, $w$ minimizes $F$ in $\mathcal D$.

By Lagrange multiplier method we have the existence of constants $\lambda_1$ and $\lambda_2$ such that for all $\varphi\in W^{1,2}(M)$ it holds
\begin{equation}\label{eq020}
    \int_M(2\langle \nabla_g w,\nabla_g\varphi \rangle+\lambda_1\varphi)dv+\lambda_2\int_{\partial M}\kappa\varphi e^wda=0,
\end{equation}
which implies that $w$ is a weak solution of $\Delta_g w=\lambda_1$ in $M$ and $\partial w/\partial\nu_g=\lambda_2\kappa e^w$ on $\partial M$. By regularity one finds that $w$ is a smooth function. Choosing $\varphi\equiv 1$ in \eqref{eq020} we obtain $\lambda_1=0$, and choosing $\varphi=e^{-w}$ we get $\lambda_2<0$, since $\int_{\partial M}\kappa da<0$. 
This implies that the function $u:=w+\log(-\lambda_2)$ satisfies $\Delta_g u=0$ in $M$ and $\partial u/\partial\nu_g=\kappa e^u$ on $\partial M$.

The item (b) follows in an analogous way.  Use the same functional, but now defined in the space
$$\mathcal D:=\left\{w\in W^{1,2}(M):\int_{M}K e^wda=0\quad\mbox{ and }\quad\int_Mwdv=0\right\}.$$
\end{proof}

The main result in this case reads as follows.

\begin{theorem} \label{PCzerocarac}
Let $(M,g)$ be a compact Riemannian surface with boundary and $\chi(M)=0$. Let $\kappa\in C^\infty(\partial M)$ and $K\in C^\infty(M)$.
\begin{itemize}
    \item[(a)] Suppose $K_g=0$. $\kappa\in PC^0(g)$ if and only if either $\kappa\equiv 0$ or $\kappa$ changes sign and 
    $\int_{\partial M}\kappa e^{-v}da<0,$
    where   $v$ is such that $\Delta_gv=0$ in $M$ and  $\partial v/\partial \nu_{g}=\kappa_g$ on $\partial M$.
    \item[(b)] Suppose $\kappa_g=0$.  $K\in PC(g)$ if and only if either $K\equiv 0$ or $K$ changes sign and $\int_{M}K e^{2v}dv<0,$ where   $v$ is such that  $\Delta_g v=K_g$ in $M$ and $\partial v/\partial\nu_g=0$ on $\partial M$.
\end{itemize}
\end{theorem}

\begin{proof}
Let us prove item (a). Using Gauss-Bonnet Theorem we find that $\int_{\partial M}\kappa_gda=0$. Let $v\in C^\infty(M)$ be the solution of the equation $\Delta_gv=0$ in $M$ and $\partial v/\partial\nu_g=\kappa_g$ on $\partial M$, which is given by \cite[Lemma 3.1]{phdthesis}.

Since $\kappa\in PC^0(g)$,  there exists $u\in C^\infty(M)$ such that
\begin{equation}\label{eq025}
\Delta_g u=0  \text { in } M\quad\mbox{ and }\quad \dfrac{\partial u}{\partial \nu_{g}}+\kappa_{g} =\kappa e^{u}  \text { on } \partial M.
\end{equation}
Define $w:=v+u$. Thus
\begin{equation}\label{eq024}
\Delta_g w=0 \text { in } M\quad\mbox{ and }\quad \dfrac{\partial w}{\partial \nu_{g}} =\kappa e^{-v}e^{w}  \text { on } \partial M.
\end{equation}
The inequality $\int_{\partial M}\kappa e^{-v}da<0,$ follows by a similar way as \eqref{eq026}.

Conversely, if $\kappa\equiv 0$, then the metric $e^{-2v}g$ belongs to $PC^0(g)$. If $\kappa$ changes sign and $\int_{\partial M}\kappa e^{-v}da<0$ holds, then by Proposition \ref{propn=2} the problem \eqref{eq024} has a solution $w$. Therefore, the function $u:=w-v$ satisfies \eqref{eq025}.

The item (b) follows in a similar manner.
\end{proof}

\subsubsection{{\bf Case 2:} $\chi(M)<0$}

 To begin with the study of this case, given constants $K_0$ and $\kappa_0$ consider the two following equations: 
\begin{equation}\label{confchangeK=0}
\Delta_g u =0 \text { in } M\quad\mbox{ and }\quad\dfrac{\partial u}{\partial \nu_{g}}+\kappa_0 =\kappa e^{u}  \text { on } \partial M
\end{equation}
and
\begin{equation}\label{confchangegeod=0}
-\Delta_g u+K_0 =Ke^{2u}  \text { in } M\quad\mbox{ and }\quad \dfrac{\partial u}{\partial \nu_{g}}=0  \text { on } \partial M.
\end{equation}

\begin{proposition}\label{solutionpointwisen=2}Let $(M,g)$  be a compact Riemannian surface with nonempty boundary.
\begin{itemize}
    \item[(a)] Suppose $\kappa_0<0$. If  \eqref{confchangeK=0} has a solution $u$ for some function $\kappa\in C^{\infty}(\partial M)$, then  the unique solution of
    \begin{equation}\label{solution 1n=2}
         \Delta_g \varphi=0 \mbox{ in } M\quad\mbox{with}\quad \frac{\partial \varphi}{\partial \nu_g}-\kappa_0\varphi=-\kappa  \mbox{ on } \partial M
    \end{equation}
is positive and $\int_{\partial M}\kappa da<0$. Moreover, there exists $\kappa\in C^\infty(\partial M)$ with $\int_{\partial M}\kappa da<0$ such that the problem \eqref{confchangeK=0} has no solution.
    \item[(b)] Suppose $K_0<0$. If  \eqref{confchangegeod=0} has a solution for some  function $K\in C^{\infty}( M),$ then  the unique solution of 
     \begin{equation}\label{solution 2n=2}
    \Delta_g\varphi+K_0\varphi= K\mbox{ in } M\quad\mbox{with}\quad \frac{\partial \varphi}{\partial \nu_g}=0  \mbox{ on } \partial M
        \end{equation}
is positive and $\int_{ M}Kdv<0.$ Moreover, there exists  $K\in C^\infty(M)$ with $\int_{M}K dv<0$ such that the problem \eqref{confchangegeod=0} has no solution.
\end{itemize}
\end{proposition}

\begin{proof}

We prove the item (a). Set $v=e^{-u}$. Thus \eqref{confchangeK=0} becomes
$$\Delta_g v=\frac{|\nabla_g v|^2}{v}   \text { in } M \quad\mbox{ and }\quad\frac{\partial v}{\partial \nu_g}-\kappa_0 v=-\kappa   \text { on } \partial M.$$

Let $\varphi$ be the unique solution of \eqref{solution 1n=2}. Setting $\Psi=\varphi-v,$ we obtain $ \Delta_g \Psi<0$ in $M$ and $\partial \Psi/\partial \nu_g=\kappa_0 \Psi$  on $\partial M,$ which implies that  the minimum of $\Psi$ is achieved at the boundary. Moreover, at the minimum point it holds $\partial \Psi/\partial \nu_g\leq 0$.   Thus $\Psi\geq 0$ at this point, and so $\varphi\geq v>0$. Observe also that  $\int_{\partial M}\kappa da<0$ follows  directly by integrating \eqref{solution 1n=2}.

For the second part,  it is enough to find $\kappa$ such that the solution of \eqref{solution 1n=2} is negative somewhere.  Let $\psi\in C^{\infty}(\partial M)$  be  any nontrivial function with $\int_{\partial M}\psi da=0$. Consider its harmonic extension to $M$, which we still denote it by $\psi.$ Let $\alpha>0$ be a constant such that $\varphi:=\psi+\alpha$ is negative somewhere. Defining 
$$
\kappa:=-\frac{\partial \psi}{\partial \nu_g} +\kappa_0(\psi+\alpha),
$$ 
we find that $\Delta_g\varphi=0$ in $M$, $\partial\varphi/\partial\nu_g-\kappa_0\varphi=-\kappa$ on $\partial M$ and $\int_{\partial M} \kappa da=\kappa_0\alpha\int_{\partial M}da<0$, since $\kappa_0<0$.

Item (b) follows similarly. Setting $v=e^{-2u}$, we obtain that \eqref{confchangegeod=0} becomes
$$\Delta_g v=\frac{|\nabla_g v|^2}{v^2}-2K_0v+2K\mbox{ in }M\quad\mbox{ and }\quad\frac{\partial v}{\partial \nu_g} =0\mbox{ on }\partial M.
    $$
Setting $\Psi=2\varphi-v$, we obtain $\Delta_g\Psi+2K_0\Psi=-|\nabla_g v|^2/v<0$ in $M$ and $\partial\Psi/\partial\nu_g=0$ on $\partial M$. Thus, we can argue as before. If the minimum of $\Psi$ is interior, then $\Delta_g\Psi\geq 0$ at this point, and so $-2K_0\Psi>\Delta_g\Psi\geq 0$. Therefore, since $K_0<0$, we get $2\varphi>v>0$. Suppose the minimum point $p$ of $\Psi$ is in the boundary. If $\Psi(p)\geq 0$, then we get the result. If $\Psi(p)<0$, using the Hopf's Lemma we obtain $\partial\Psi/\partial\nu_g<0$ at $p$, which is a contradiction.

The rest of the proof follows as before.
\end{proof}

The main result of this case is stated in the following two theorems.

\begin{theorem}\label{teo001}
Let $(M,g)$  be a compact Riemannian surface with nonempty boundary, $\chi(M)<0$ and vanishing Gaussian curvature.
\begin{itemize} 
    \item[(a)]  If $\kappa \in  \mbox{PC}^0(g),$ then
\begin{equation}\label{quaseGB}
\int_{\partial M} \kappa e^{-v} da<0,
    \end{equation}
where $v$ is a solution of $\Delta_g v=0$ in $M$ and  $\partial v/\partial \nu_g=\kappa_g-\overline{\kappa}_g$ on $\partial M.$ Moreover,  the unique solution $\varphi$ of
 \begin{equation*}
         \Delta_g \varphi=0 \mbox{ in } M\quad\mbox{with}\quad \frac{\partial \varphi}{\partial \nu_g}-\overline{\kappa}_g\varphi=-\kappa e^{-v} \mbox{ on } \partial M
    \end{equation*}
must be positive. Here $\overline\kappa_g:=Area(\partial M)^{-1}\int_{\partial M}\kappa_gda$.
\item[(b)] There exists $\kappa\in C^\infty(\partial M)$ satisfying  \eqref{quaseGB}  such that $\kappa\notin  \mbox{PC}^0(g).$ 
\item[(c)]$\kappa \in  \mbox{PC}^0(g)$ if and only if there exists a solution of
\begin{equation}\label{eq028}
\Delta_g u\leq 0\mbox{ in }M\quad\mbox{ and }\quad
\dfrac{\partial u}{\partial \nu_{g}}+\kappa_{g} \geq \kappa e^{u} \text { on } \partial M.
\end{equation}
\end{itemize}
\end{theorem}

\begin{proof}
If $\kappa \in  \mbox{PC}^0(g),$ then there exists a function $u$ satisfying \eqref{eq025}. By \cite[Lemma 3.1]{phdthesis} there exists $v$ such that $\Delta_g v=0$ in $M$ and $\partial v/\partial\nu_g=\kappa_g-\overline\kappa_g$ on $\partial M$. Thus, if we define $w:=v+u$ then we get
\begin{equation}\label{eq027}
\Delta_g w =0\mbox{ in }M\quad\mbox{ and }\quad \frac{\partial w}{\partial \nu_g}+\overline{\kappa}_g =(\kappa e^{-v})e^{w}\mbox{ on }\partial M.
\end{equation}
By the Gauss-Bonnet Theorem we have $\overline\kappa_g<0$. Thus item (a) and (b) follow from item (a) of the Proposition \ref{solutionpointwisen=2}. For item $(c)$,  note that $\kappa\in \mbox{PC}^0(g)$ if and only if there exists a solution of \eqref{eq027}. Also, observe that if $u$ is a small constant, then
$$
\dfrac{\partial u}{\partial \nu_{g}}+\overline\kappa_{g}-\kappa e^{u}< 0,
$$
that is, a small constant is a lower solution for the problem \eqref{eq027}. Otherwise, a solution to \eqref{eq028} give rise a upper solution for \eqref{eq027}. Thus, the result follows by \cite[Theorem 2.3.1]{Sa}.
\end{proof}


\begin{theorem}\label{teo004}
Let $(M,g)$  be a compact Riemannian surface with nonempty boundary, $\chi(M)<0$ and geodesic boundary.
\begin{itemize} 
    \item[(a)]  If $K \in  \mbox{PC}(g),$ then
\begin{equation}\label{quaseGB2}
\int_{ M} K e^{2v} dv<0,
    \end{equation}
where $v$ is a solution of $\Delta_g v=K_g-\overline{K}_g$ in $M$ and $\partial v/\partial \nu_g=0$ on $\partial M.$ Moreover, the unique solution $\varphi$ of
 \begin{equation*}
         \Delta_g \varphi+\overline{K}_g=Ke^{2v} \mbox{ in } M\quad\mbox{and}\quad \frac{\partial \varphi}{\partial \nu_g}=0  \mbox{ on } \partial M
    \end{equation*}
must be positive. Here $\overline K_g:=Vol(M)^{-1}\int_MK_gdv$.
\item[(b)] There exists $K\in C^\infty(M)$ satisfying \eqref{quaseGB2} such that $K\notin  \mbox{PC}(g).$ 
\item[(c)]$K \in  \mbox{PC}(g)$ if and only if there exists a solution of 
\begin{equation*}
\Delta_g u-K_{g}\leq -Ke^{2u}\mbox{ in } M\quad\mbox{ and }\quad \dfrac{\partial u}{\partial \nu_{g}} \geq 0  \text { on } \partial M.
\end{equation*}

\end{itemize}
\end{theorem}

\begin{remark}
It is important to mention here that the Gaussian curvature of the metric $g$ in Theorem \ref{teo004} is allowed to be positive somewhere,  in contrast to  \cite[Theorem 1.1]{MR4517687}, where $K_g$ needs to be negative.
\end{remark}

\subsubsection{{\bf Case 3:} $\chi(M)>0$}\label{sec003}

Let us recall the following theorem, which uses an idea introduced by \cite{KW1}.
\begin{theorem}[P. Liu \cite{phdthesis-liu}]
Let $(M,g)$ be a compact Riemannian surface with boundary, then there exists a constant $C$, which depends only on the geometry of $M$, such that for all $u\in W^{1,2}(M)$
\begin{equation}\label{eq035}
    \log\int_{\partial M}e^uda\leq\frac{1}{4\pi}\int_M|\nabla u|^2dv+\int_{\partial M}uda+C.
\end{equation}
The value $1/4\pi$ is sharp.
\end{theorem}

\begin{theorem}\label{teo005}
Let $(M,g)$ be a compact Riemannian surface with nonempty boundary, $\chi(M)=1$ and vanishing Gaussian curvature. Let $\kappa\in C^\infty(\partial M)$.
\begin{itemize}
    \item[(a)] If $\kappa\in \mbox{PC}^0(g)$, then $\kappa$ is positive somewhere.
    \item[(b)] Suppose $\kappa_g=c>0$ is constant and $\kappa$ is positive somewhere. Let $\gamma\in(0,1)$. Then there exists a smooth solution $u$ of the problem
    \begin{equation}\label{eq039}
        \Delta_gu=0\mbox{ in }M\quad\mbox{ and }\quad\frac{\partial u}{\partial \nu_g}+c=\kappa e^{\gamma u}\mbox{ on }\partial M
    \end{equation} 
    such that  $\kappa e^{(\gamma-1)u}\in PC^0(g)$.
\end{itemize}
\end{theorem}
\begin{proof}
Item (a) follows by the Gauss-Bonnet Theorem. To prove item (b), consider the functional $J:\mathcal D\rightarrow\mathbb R$ given by
$$J(u)=\frac{1}{2}\int_M|\nabla u|^2dv+c\int_{\partial M}uda,$$
where 
$\mathcal D:=\left\{u\in W^{1,2}(M):\int_{\partial M}\kappa e^{\gamma u}=2\pi\right\}.$
Since $\kappa$ is positive somewhere, it is not difficult to show that $\mathcal D$ is nonempty. By the Gauss-Bonnet Theorem, it holds that $c \cdot \text{Area}(\partial M) = 2\pi$, which implies 
\begin{equation}\label{eq040}
c\int_{\partial M}uda=2\pi\overline u,    
\end{equation}
where $\overline u:=\mbox{Area}(\partial M)^{-1}\int_{\partial M}uda$. Replacing $u$ by $w+\overline u$ in $\int_{\partial M}ke^{\gamma u}=2\pi$, with $w:=u-\overline u$, we find
 $$\overline u=-\frac{1}{\gamma}\log\int_{\partial M}ke^{\gamma w}+\frac{1}{\gamma}\log(2\pi).$$
 By \eqref{eq040} one sees that $J$ can be expressed as
$$J(u)=\frac{1}{2}\int_M|\nabla u|^2dv-\frac{2\pi}{\gamma}\log\int_{\partial M}\kappa e^{\gamma w}da+\frac{2\pi}{\gamma}\log(2\pi).$$

By \eqref{eq035} we find that
$$J(u)\geq \frac{1-\gamma}{2}\int_M|\nabla u|^2+\mbox{const}.$$
This implies that $J$ is bounded from below. Using that the trace embedding $W^{1,2}(M)\rightarrow L^q(\partial M)$ is continuous (in fact, compact) for all $q\geq 1$, as in the proof of Proposition \ref{propn=2} we can find a minimizer $u_0$ for $J$, since $1-\gamma>0$.

By Lagrange multiplier method, there exists a constant $\lambda$ such that for all $\varphi\in W^{1,2}(M)$ it holds
$$\int_M\langle \nabla u,\nabla\varphi\rangle dv+c\int_{\partial M} \varphi da+\lambda\gamma \int_{\partial M} \varphi\kappa e^{\gamma u}da=0.$$
By choosing $\varphi\equiv 1$, we get that $\lambda\gamma=-1$. Thus $u$ is a solution to \eqref{eq039}.
\end{proof}

In \cite{MR1417436} the authors found sufficient conditions to assure that $k\in PC^0(g)$. They constructed a new functional that satisfies the Palais Smale condition in a suitable space and possesses the same critical set as $J$.

\begin{theorem}\label{teo006}
Let $(M,g)$ be a compact Riemannian surface with nonempty boundary, $\chi(M)=1$ and geodesic boundary. Let $K\in C^\infty(M)$
\begin{itemize}
    \item[(a)] If $K\in \mbox{PC}(g)$, then $K$ is positive somewhere.
    \item[(b)] Suppose $K_g=k>0$ is constant and $K$ is positive somewhere. Let $\gamma\in(0,1)$. Then there exists a smooth solution $u$ of the problem
    $$-\Delta_gu+k=Ke^{2\gamma u}\mbox{ in }M\quad\mbox{ and }\quad\frac{\partial u}{\nu_g}=0$$
    such that 
    $Ke^{2(\gamma-1)u}\in \mbox{PC}(g)$.
\end{itemize}
\end{theorem}

The proof is similar to the previous one. Use the functional
$$G(u)=\int_M\left(\frac{1}{2}|\nabla u|^2+ku\right)dv$$
defined in
$\mathcal H:=\left\{u\in W^{1,2}(M):\int_MKe^{2\gamma u}=2\pi\mbox{ and }\frac{\partial u}{\partial\nu_g}=0\right\}.$
For all $u\in W^{1,2}(M)$ it holds
$$\log\int_Me^udv\leq\frac{1}{8\pi}\int_M|\nabla u|^2dv+\mbox{Area}(M)^{-1}\int_Mudv+C,$$
for some constant $C$ which depends only on $M$, see \cite[Corollary 3.5]{phdthesis}. Using this and the fact that $\gamma\in(0,1)$ we can find a minimizer for $G$ as before.

\begin{remark}
 Theorems \ref{teo005} and \ref{teo006} are not easy to deal with in the case $\gamma=1$. If the surface is closed, this corresponds to the well-known Nirenberg problem, which remains an open problem  and it has attracted a lot of attention in the last decades. In the preceding scenario, the functionals $J$ and $G$ remain bounded from below even when $\gamma=1$, but in general it is not possible to find a minimizer.  This conclusion is drawn from \cite[Proposition 4.5]{EOB}, as discussed in Section \ref{case_2}, which provides a nontrivial obstruction for functions to belong to PC($g$) and PC$^0$($g$).
\end{remark}

\section{Prescribed curvature problems in higher dimension}\label{sec002}

\subsection{Existence of conformally equivalent metric}\label{sec001}
Let $(M^n,g)$ be a compact connected manifold with nonempty boundary and dimension $n\geq 3$. The content of this section extend some results of \cite{KW2} to this context.

From the variational characterization of $\lambda_1(\mathcal L_g)$ and $\sigma_1(\mathcal B_g)$ in \eqref{eing1} and \eqref{eing2} it follows that  $\lambda_1(\mathcal L_g)$ is positive (negative, zero) if and only if $\sigma_1(\mathcal B_g)$ is positive (negative, zero). Also, the first eigenfunctions for problems \eqref{eing1} and \eqref{eing2} are strictly positive (or negative), see \cite[Proposition 1.3]{E3}. Moreover,  its importance follows from the fact that the  sign of $\lambda_1(\mathcal L_g)$ is uniquely determined by the conformal structure (see \cite[Proposition 1.3]{E1} and \cite[Proposition 1.2]{E3}). More precisely we have the following proposition which is the content of \cite[Proposition 1.3 and Lemma 1.1]{E1} and \cite[Proposition 1.2 and Proposition 1.4]{E3}.

\begin{proposition}[Escobar \cite{E1,E3}]\label{escobar}
Let $(M^n, g)$ be a compact Riemannian manifold with
nonempty boundary and dimension $n\geq 3$. \begin{itemize}
    \item[(a)] There exists a metric pointwise conformal to $g$ whose scalar curvature is zero and the mean curvature of the boundary does not change sign. The sign is the same of $\lambda_1(\mathcal L_g)$ which is uniquely determined by the conformal structure.
    \item[(a)] There exists a metric pointwise conformal to $g$ whose scalar curvature does not change sign and the boundary is minimal. The sign is the same of $\sigma_1(\mathcal B_g)$ which is uniquely determined by the conformal structure.
\end{itemize}
\end{proposition}

Before we prove the main results of this section, we show the following lemma.

\begin{lemma}\label{help} Let $(M^n,g)$ be a compact Riemannian manifold with boundary and dimension $n\geq 3$. Let $H\in C^{\infty}(\partial M)$ and $R\in C^{\infty}(M).$
\begin{itemize} 
\item[(a)] Suppose $R_g\equiv 0$ and $H_g\equiv H_0$ is constant.
If there exists a constant $k>0$ satisfying 
$
\min H \leq kH_0\leq \max H,
$
then $H$ is the mean curvature of a scalar flat metric conformally equivalent to $g$.
\item[(b)] Suppose $R_g\equiv R_0$ is constant and $H_g\equiv 0$.
If there exists a constant $k>0$ satisfying 
$
\min R \leq kR_0\leq \max R, $
then $R$ is scalar curvature of a  metric with minimal boundary conformally equivalent to $g.$
\end{itemize}
\end{lemma} 
 \begin{proof}
We  prove item $(a)$ since item $(b)$ is entirely similar.

It is enough to prove the result for $k=1$, since we can multiply the metric by an appropriate constant, see \eqref{confchange}.
By the Perturbation Theorem \ref{perttheo} given $\varepsilon>0$, there exists $u_{1} \in C^{\infty}(M)$ sufficiently close to $u_{0}=1$ such that 
$$\|F_2(1)-F_2(u_1)\|_{\infty}<\varepsilon,$$
and $F_2'(u_1)$ is invertible, where $F_2(1)=(0,H_0)$ and $F_2$ is defined in \eqref{eq019}.
By using the approximation Lemma \ref{appr} and the Inverse Function Theorem for Banach spaces, we present the argument as the proof of Theorem \ref{teoa} (Section \ref{proof-theorem-1-2}). This approach enables us to arrive at our intended result.
 \end{proof}

For $\lambda_1(\mathcal L_g)<0$ we have the following result.

\begin{theorem}\label{lambda<0}
Let $(M^n,g)$ be a compact Riemannian manifold with boundary,  dimension $n\geq 3$ and $\lambda_1(\mathcal L_g)<0$. Let $h\in C^\infty(\partial M)$ and $f\in C^\infty(M)$.
\begin{itemize}
    \item[(a)] $h\in CE^0(g)$ if and only if $h$ is negative somewhere.
    \item[(b)] $f\in CE(g)$ if and only if $f$ is negative somewhere.
\end{itemize}
\end{theorem}
\begin{proof}
We  only prove  item $(a)$ since item $(b)$ is entirely analogous.
Suppose $h\in CE^0(g)$. Then there exists a diffeomophism $\varphi$ and a smooth positive function $u$ such that $\overline g:=u^{\frac{4}{n-2}}\varphi^*g$ is scalar flat and has mean curvature equal to $h$. Thus $\mathcal L_{\varphi^*g}u=0$ and $\mathcal B_{\varphi^*g}u=hu^{\frac{n}{n-2}}$. Let $\psi$ be the first eigenfunction of \eqref{eing2}. This implies that
\begin{align*}
    \sigma_1(\mathcal B_{\varphi^*g})\langle \psi,u \rangle_{L^2(\partial M)} & =\langle \mathcal B_{\varphi^*g}\psi,u \rangle_{L^2(\partial M)}=\langle \psi,\mathcal B_{\varphi^*g}u \rangle_{L^2(\partial M)}\\
    &=\langle\psi,hu^{\frac{n}{n-2}} \rangle_{L^2(\partial M)}.
\end{align*}

According to Proposition \ref{escobar}, we establish that $\sigma_1(\mathcal B_{\varphi^*g})<0$. Given that both $\psi$ and $u$ are positive functions, it follows that $h$ must be negative somewhere.

Now, suppose that $h\in C^\infty(M)$ is negative somewhere. By \cite[Remark 6.4]{CV}, there exists a pointwise conformal scalar flat metric $g_0$ with mean curvature $H_{g_0}\equiv -1$. The result follows by Lemma \ref{help}.
\end{proof}

For $\lambda_1(\mathcal L_g)=0$ we have the following result.
\begin{theorem}\label{teo003}
Let $(M^n,g)$ be a compact Riemannian manifold with boundary,  dimension $n\geq 3$ and $\lambda_1(\mathcal L_g)=0$. Let $h\in C^\infty(\partial M)$ and $f\in C^\infty(M)$.
\begin{itemize}
    \item[(a)] $h\in CE^0(g)$ if and only if $h$ changes sign or is identically zero.
    \item[(b)] $h\in CE(g)$ if and only if $f$ changes sign or is identically zero.
\end{itemize}
\end{theorem}
\begin{proof}
As Theorem \ref{lambda<0},
 we   prove  item $(a)$. Let $\varphi>0$ be the eigenfunction of the problem \eqref{eing2}. 
 As established in Proposition \ref{escobar}, we have $\sigma_1(\mathcal{B}_g) = 0$. This implies the existence of a metric that is conformally equivalent to $g$, which is  scalar flat and has  minimal boundary. The remainder of the proof follows a similar structure to that of Theorem \ref{lambda<0}.
\end{proof}

Finally, when $\lambda_1(\mathcal L_g) > 0$, the following holds.
\begin{theorem}\label{lambda>0}
Let $(M^n,g)$ be a compact Riemannian manifold with boundary,  dimension $n\geq 3$ and $\lambda_1(\mathcal L_g)>0$. Let $h\in C^\infty(\partial M)$ and $f\in C^\infty(M)$.
\begin{itemize}
    \item[(a)] $h\in CE^0(g)$ if and only if $h$ is positive somewhere.
    \item[(b)] $f\in CE(g)$ if and only if $f$ is positive somewhere.
\end{itemize}
\end{theorem}
\begin{proof}
The proof is  analogous to Theorem \ref{lambda<0}, using the solution of the Yamabe Problem with boundary (see \cite{A,BS,E1,Mayer-Ndiaye,chen2010conformal}).
\end{proof}

  By Theorems \ref{lambda<0}, \ref{teo003} and \ref{lambda>0} we obtain Theorems \ref{teob} and \ref{teoc}.

\subsection{Existence and nonexistence of pointwise conformal metric}\label{pointn>2}

  In this section we extend some results of \cite{KW3} to the context of manifold with boundary. Let $(M^n,g)$ be a compact manifold with nonempty boundary and dimension $n\geq 3$. The existence of a scalar flat metric in the conformal class $[g]$  with mean curvature  $H\in C^{\infty}(\partial M)$  is equivalent to the existence of a positive solution to the problem:
\begin{equation}\label{confchangeR=0}
\mathcal L_g(u)=0\mbox{ in }M\quad\mbox{ and }\quad\mathcal B_g(u)=Hu^{\frac{n}{n-2}}\mbox{ on }\partial M.
\end{equation}

Similarly, the existence of a metric in  the conformal class $[g]$  with scalar curvature $R\in C^{\infty}(M)$ and minimal boundary is equivalent to the existence of a positive solution to the problem
\begin{equation}\label{confchangeH=0}
\mathcal L_g(u)=Ru^{\frac{n+2}{n-2}}\mbox{ in }M\quad\mbox{ and }\quad\mathcal B_g(u)=0\mbox{ on }\partial M.
\end{equation}
Here $\mathcal L_g$ and $\mathcal B_g$ are defined in \eqref{confchange}. If such a solution exists, then $H\in\mbox{PC}^0(g)$ and $R\in\mbox{PC}(g)$, respectively.
  
  Again, our analysis will be divided in three cases.
  
\subsubsection{ {\bf Negative curvature case:} $\lambda_1(\mathcal L_g)<0$}
 As pointed out in Section \ref{sec001}, $\sigma_1(\mathcal{B}_g) < 0$ if and only if $\lambda_1(\mathcal{L}_g) < 0$ (see \cite{E1,E3}). This implies the existence of a pointwise conformal scalar-flat metric with a constant negative mean curvature, as well as a pointwise conformal metric with constant negative scalar curvature and minimal boundary (see, for instance, \cite[Remark 6.4]{CV}). We begin with the following theorem.

\begin{theorem}\label{solutionpointwise} Let $(M^n,g)$  be a compact manifold with nonempty boundary and dimension $n\geq 3$.
\begin{itemize}
    \item[(a)] Suppose  that $g$ is scalar flat metric with  constant mean curvature $H_0<0.$ If \eqref{confchangeR=0} has a positive solution for some $H\in C^{\infty}(\partial M),$ then the unique solution of the problem
     \begin{equation}\label{solution 1}
 \Delta_g \varphi=0  \text { in } M    \quad\mbox{ and }\quad
\frac{\partial \varphi}{\partial \nu_g}-H_0\varphi=-H  \text { on } \partial M
\end{equation}
must be positive and $\int_{\partial M}Hda<0$. Moreover, there exists a function $H \in C^{\infty}(\partial M)$ with $\int_{\partial M} H da<0$ such that $H\not\in\mbox{PC}^0(g)$. 
    \item[(b)] Suppose that $g$  has constant  scalar curvature $R_0<0$ and minimal boundary. If \eqref{confchangeH=0} has a positive solution for some $R\in C^{\infty}( M),$ then  the unique solution of the problem
    \begin{equation}\label{solution 2}
    \Delta_g \varphi+\frac{R_0}{n-1}\varphi= \frac{R}{n-1} \text { in } M\quad\mbox{ and }\quad \frac{\partial \varphi}{\partial \nu_g}=0  \text { on } \partial M
\end{equation}
must be positive and $\int_{ M}Rdv<0$. Moreover, there exists a function $R \in C^{\infty}(M)$ with $\int_{M} R dv<0$ such that $R\not\in\mbox{PC}(g)$.
\end{itemize}
\end{theorem}
\begin{proof}
Let us prove item (a). If $v=u^{-\frac{2}{n-2}} $,   then \eqref{confchangeR=0} is equivalent to
\begin{equation*}
2\Delta_g v-n\frac{|\nabla_g v|^2}{v}=0  \text { in } M\quad\mbox{ and }\quad \frac{\partial v}{\partial \nu_g}-H_0 v=-H  \text { on } \partial M.
\end{equation*}
  
  Let $\varphi$ be the unique solution of \eqref{solution 1}.   Setting $\Psi:=\varphi-v,$ we obtain that $ \Delta_g \Psi< 0$ in $M$ and $\partial\Psi/\partial\nu_g=H_0 \Psi$ on $\partial M$. As in  Proposition \ref{solutionpointwisen=2}, we obtain  $\varphi\geq v>0$. Finally, observe that the weaker conditions $\int_{\partial M}Hda<0$ is obtained by integration of \eqref{solution 1}. 

In an analogous way, the item (b) is obtained considering $v=u^{-\frac{4}{n-2}}$. Then  \eqref{confchangeH=0} is equivalent to
$$\Delta_gv+\frac{R_0}{n-1}v=\frac{1}{n-1}R+\frac{n+2}{4}\frac{|\nabla_g v|^2}{v}\mbox{ in }M\quad\mbox{ and }\quad \frac{\partial v}{\partial\nu_g}=0\mbox{ on }\partial M.$$
Thus, if we define $\Psi:=\varphi-v$ we get $\Delta_g\Psi+\frac{R_0}{n-1}\Psi<0$ in $M$ and $\partial\Psi/\partial\nu_g=0$ on $\partial M$. Arguing as in the Proposition \ref{solutionpointwisen=2} we obtain the result.
\end{proof}


Also, we can prove the following result.

\begin{proposition}\label{prop001}
Let $(M^n,g)$  be a compact manifold with nonempty boundary and dimension $n\geq 3$.
\begin{itemize}
    \item[(a)]  Suppose that $g$ is a scalar-flat metric with constant mean curvature $H_0<0$. Let $H,H_1\in C^\infty(\partial M)$. If  $H\in\mbox{PC}^0(g)$ and either $H_1\leq H$ or $H_1=k_1H$, for some constant $k_1>0$, then $H_1\in\mbox{PC}^0(g)$. 
    
    \item[(b)] Suppose that $g$ is a metric of constant  scalar curvature $R_0<0$ and minimal boundary. Let $R,R_1\in C^\infty(M)$. If $R\in\mbox{PC}(g)$ and either $R_1\leq R$ or $R_1=k_2R$, for some constant $k_2>0$, then $R_1\in\mbox{PC}(g)$.
\end{itemize}
\end{proposition}

\begin{proof}
As before, we  prove item (a). First we prove the following claim.

\medskip

\noindent\textbf{Claim:}  If $\sigma_1(\mathcal{B}_g)<0$, then one can  find a positive lower solution of \eqref{confchangeR=0}.

Let $u_-=\gamma \varphi,$ where $\varphi$ is the first normalized eigenfunction of the problem \eqref{eing2}, which is positive, and $\gamma$ is a positive constant. If $H_1\in C^\infty(M)$, then $\mathcal L_g(u_-)=0$ and
$
\mathcal{B}_g(u_-)-H_1(u_-)^{\frac{n}{n-2}}=\gamma\varphi(\sigma_1(\mathcal B_g)-H_1(\gamma\varphi)^{\frac{4}{n-2}}).
$
Since $H_1$ and $\varphi$ are bounded and $\gamma$ can be chosen sufficiently small, the $u_-$ is a lower solution of \eqref{confchangeR=0}, with $H$ replaced by $H_1$.

\medskip

If $H\in\mbox{PC}^0(g)$, then there exists a positive smooth function $u^+$ such that $\mathcal L_g(u^+)=0$ and $\mathcal B_g(u^+)=H(u^+)^{\frac{n}{n-2}}$. Assuming $H_{1} \leq H$, we obtain that $\mathcal B_g(u^+)\geq H_1(u^+)^{\frac{n}{n-2}}$, which implies that $u^+$ is a upper solution of \eqref{confchangeR=0}, with $H$ replaced by $H_1$. For $\gamma>0$ small enough we have $u_-\leq u^+$. By  Theorem 2.3.1 of \cite{Sa} we conclude that $H_1\in\mbox{PC}^0(g)$.

Furthermore, if we have $H_1 = k_1 H$, where $k_1 > 0$ is a constant, and since $H \in \mathcal{PC}(g)$, it is sufficient to choose an appropriate scalar multiple of the metric $g_0$ such that both $R_{g_0} = 0$ and $H_{g_0} = H$.
\end{proof}

In fact, we can say more.

\begin{proposition}\label{prop002}
Let $(M^n,g)$  be a compact manifold with nonempty boundary and dimension $n\geq 3$.
\begin{itemize}
   \item[(a)]  Suppose that $g$ is scalar-flat with constant mean curvature
   $H_0$. Given $H\in C^\infty(\partial M)$ with $\int_{\partial M}Hda<0$, there exists a constant $k(H)<0$ such that $H\in\mbox{PC}^0(g)$, provided  $k(H)< H_0<0$.
   \item[(b)]   Suppose that $g$ has minimal boundary and constant scalar curvature $R_0<0$. Given $R\in C^\infty(M)$ with $\int_M R dv<0$, there exists a constant $k(R)<0$  such that $R\in\mbox{PC}(g)$, provided $k(R)< R_0<0$.
    \end{itemize}
\end{proposition}
\begin{proof}
Let us prove item (a). In order to demonstrate that $H\in \mathrm{PC}^0(g)$, we need to find a positive solution $u$ to \eqref{confchangeR=0}, which takes the form:
\begin{equation}\label{eq030}
\Delta_g u = 0 \text{ in } M \quad \text{and} \quad \frac{2}{n-2}\frac{\partial u}{\partial\nu_g} + H_0 u = H u^{n/(n-2)} \text{ on } \partial M,
\end{equation}
with $H_0<0$ constant. As the proof of Proposition \ref{prop001} there exists a positive lower solution $u_-$, which can be made small enough. Thus, we need only to find an upper solution $u^+$. The result will follow by \cite[Theorem 2.3.1]{Sa}.

To find an upper solution to \eqref{eq030} let us show that the following auxiliary problem has a solution:
\begin{equation}\label{eq029}
2\Delta_g v-n\frac{|\nabla_g v|^2}{v}\geq 0  \text { in } M\quad\mbox{ and }\quad \frac{\partial v}{\partial \nu_g}-H_0 v\leq -H  \text { on } \partial M.
\end{equation}



Let  $\Psi\in C^{\infty}(M)$  be a solution of the problem $\Delta_g \Psi=-\frac{1}{\mbox{vol}(M)}\int_{\partial M}Hda$ in $M$ and $\partial \Psi/\partial \nu_g=2\overline{H}-H$ on $\partial M$ (see \cite[Lemma 3.1]{phdthesis}), where $\overline H=\mbox{Area}(\partial M)^{-1}\int_{\partial M}Hda$.
Set $v:=\Psi+\gamma$, where the constant $\gamma$ is taken so large that  $v>0$ and 
$$2\Delta_gv-n\frac{|\nabla_g v|^2}{v}=2\Delta_g\Psi-n\frac{|\nabla_g \Psi|^2}{\Psi+\gamma}>0.$$
Note that such $\gamma$ does exist since $\Delta\Psi>0$. Since $\overline H<0$, there exists $H_0<0$ so that
$
2\overline{H}\leq H_0v.
$
Thus, $v$ is a solution to $\eqref{eq029}$. Finally, set $u^+:=v^{-\frac{n-2}{2}}$ and note that
$$\Delta_gu^+=-\frac{n-2}{4}\left(2\Delta_gv-n\frac{|\nabla_gv|^2}{v}\right)\leq 0$$
and
$$\frac{2}{n-2}\frac{\partial u^+}{\partial\nu_g}+H_0u^+=-v^{-\frac{n}{2}}\left(\frac{\partial v}{\partial\nu_g}-H_0v\right)\geq v^{-\frac{n}{2}}H=Hu^{\frac{n}{n-2}}.$$
This finishes item (a).

Item (b) is similar. Consider the auxiliary problem
\begin{equation}\label{eq031}
   \Delta_gv+\frac{R_0}{n-1}v-\frac{R}{n-1}\geq 0 \mbox{ in }M\quad\mbox{ and }\quad\frac{\partial v}{\partial\nu_g}\leq 0\mbox{ on }\partial M.
\end{equation}


Let $\Psi$ be a solution of $\Delta_g \Psi=(R-\overline{R})/(n-1)$ in $M$ with $\partial \Psi/\partial \nu_g=0$ on $\partial M$, where $\overline{R}=\mbox{Vol}(M)^{-1}\int_MRdv<0$. Set 
$v:=\Psi+\gamma$,  where the constant $\gamma$ is so large that  $v>0$ and 
]\begin{equation}\label{eq041}
-\frac{\overline{R}}{2(n-1)}\geq \frac{n+2}{4} \frac{|\nabla_g v|^{2}}{v}=\frac{n+2}{4} \frac{|\nabla_g \Psi|^{2}}{\Psi+\gamma}.    
\end{equation}
Also, there exists $R_0 < 0$ so that $\overline{R} \leq R_0(\Psi + \gamma)$. With this choice, we conclude that  $v$ is a solution to \eqref{eq031}. Now define $u^+:=v^{-\frac{n-2}{4}}$. Using \eqref{eq031} and \eqref{eq041}, we obtain
\begin{align*}
    \frac{4(n-1)}{n-2}\Delta_gu^+&-R_0u^+  =-(n-1)v^{-\frac{n+2}{4}}\left(\Delta_gv-\frac{n+2}{4}\frac{|\nabla_g v|^2}{v}\right)\\
    & \leq -(n-1)v^{-\frac{n+2}{4}}\left(\frac{R}{n-1}-\frac{R_0}{n-1}v-\frac{n+2}{4}\frac{|\nabla_gv|^2}{v}\right)\\
    &\leq- Rv^{-\frac{n+2}{4}}=-R(u^+)^{\frac{n+2}{n-2}}
\end{align*}
in $M$ and $\partial u^+/\partial\nu_g\geq 0$ on $\partial M$. This implies that $u^+$ is an upper solution to \eqref{confchangeH=0}. The existence of a lower solution follows as item (a) of Proposition \ref{prop001}.
\end{proof}

\subsubsection{{\bf  Positive and zero curvature case:} } 
\label{case_2}
The cases where $\lambda_1(\mathcal{L}_g) > 0$ and $\lambda_1(\mathcal{L}_g) = 0$ have already been studied. If $\lambda_1(\mathcal{L}_g) > 0$, then there are nontrivial obstructions for a smooth function to belong to $PC^0(g)$ or $PC(g)$.

Indeed, for the unit ball $B^n \subset \mathbb{R}^n$, we have a nontrivial Kazdan-Warner type obstruction discovered by Escobar in Proposition 4.6 of \cite{EOB}. This result states that if $H \in C^\infty(\partial M)$ is the mean curvature of a pointwise conformal metric, then
$$
\int_{\partial B^{n}}\left\langle X, \nabla_g H\right\rangle da=0,
$$
where $X$ is a conformal Killing vector field on $\partial B^n$. This condition implies that problem \eqref{confchangeR=0} has no solutions for $H = Ax + B$, where $A \neq 0$. Similar obstructions for the Euclidean hemisphere $\mathbb{S}_+^n \subset \mathbb{R}^{n+1}$ were obtained in \cite{Ham}.

Finally, for $\lambda_1(\mathcal{L}_g) = 0$, J. Xu gives results on prescribing nonzero scalar and mean curvature problems  in \cite{xu2301conformal}.

\bibliography{references.bib}
\bibliographystyle{acm}

\end{document}